\documentclass[10pt,a4paper,reqno]{amsart}

\usepackage{amsmath, amsthm, amsfonts, amssymb, mathrsfs}
\usepackage{graphicx}
\usepackage{mathtools}
\usepackage[update,prepend]{epstopdf} 
\usepackage{color}
\usepackage{xcolor}
\usepackage{soul}

\numberwithin{equation}{section}

\usepackage{marginnote}
\usepackage{enumerate}
\usepackage{marginnote}
\usepackage{hyperref}
\newcommand{\NN}{\mathbb{N}}
\newcommand{\RR}{{\mathbb{R}}}
\newcommand{\CC}{{\mathbb{C}}}

\newcommand{\cF}{{\mathcal{F}}}
\newcommand{\cH}{{\mathcal{H}}}

\newcommand{\Dom}{{\operatorname{Dom}}}

\newcommand{\eps}{\varepsilon}

\newcommand{\ov}{\overline}

\newcommand\la{\lambda}

\renewcommand{\d}{{\textrm{d}}}

\renewcommand\eps{\varepsilon}

\newcommand\sd{\sigma_{\rm disc}}
\newcommand\sess{\sigma_{\rm ess}}

\newcommand\mydot{\,\cdot\,}
\newcommand\ds{\displaystyle}

\newcommand\wh{\widehat}

\theoremstyle{plain}

\newtheorem{theorem}{Theorem}
\newtheorem{lemma}[theorem]{Lemma}

{Proposition}

\theoremstyle{definition}

\newtheorem{assumption}{Assumption}

\theoremstyle{remark}
\newtheorem{remark}{Remark}
{Notation}

 \theoremstyle{definition}
 
 \theoremstyle{remark}

 \numberwithin{equation}{section}

\title[The spectrum of the spin-boson Hamiltonian for arbitrary coupling]{Spectral analysis of the spin-boson Hamiltonian with two photons for arbitrary coupling}

\begin{document}

\thanks{The author is grateful to Professor Alexander V.~Sobolev for fruitful discussions, reading the preliminary version of this manuscript and making useful suggestions and thanks the Department of Mathematics at University College London for the kind hospitality. The financial support of the \emph{Swiss National Science Foundation}, SNF, through the Early Postdoc.Mobility grant No.\ $168723$ is gratefully acknowledged.} 

\date{\today}

\author{Orif \,O.\ Ibrogimov}
\address{Department of Mathematics, 
	University College London,
	Gower Street, London,
	WC1E 6BT, ~UK}

\email{o.ibrogimov@ucl.ac.uk, o.ibrogimov@gmail.com}





\begin{abstract}
We study the spectrum of the spin-boson model with two photons in~$\RR^d$ for arbitrary coupling $\alpha>0$. It is shown that the discrete spectrum is finite and the essential spectrum consists of a half-line, the bottom of which is a unique zero of a simple Nevanlinna function. Besides the simplicity and more abstract nature of our approach, the main novelty is the achievement of these results under minimal regularity conditions on the photon dispersion and the coupling function. 
\end{abstract}

\maketitle

\section{Introduction}\label{sec:intro}
The spin-boson model is a well-known quantum-mechanical model which describes the interaction between a two-level atom and a photon field. We refer to \cite{Legget-1987} and \cite{Huebner-Spohn-1995-review} for excellent reviews from physical and mathematical perspectives, respectively. 

Despite the formal simplicity of the spin-boson model (from the physics viewpoint), its dynamics is rather complicated and rigorous spectral and scattering results are usually very difficult to obtain, especially in the case when the number of photons is unbounded. For weak coupling, starting with the pioneering work of H\"ubner and Spohn~\cite{Huebner-Spohn-1995} spectral and scattering properties of the full spin-boson model as well as of its finite photon approximations have been investigated extensively; see, for example, \cite{Spohn-CMP-1989,Zhukov-Milnos-1995, Huebner-Spohn-1995, Minlos-Spohn-1996, Gerard-1996,  Arai-Hirokawa-1997, deMonvel-Sahbani-1998, Skibsted-1998, Arai-2000, Hirokawa-2001RMP, Angelescu-Minlos-Ruiz-Zagrebnov-2008, Abdesselam-2011, Hasler-Herbst-2011, DeRoeck-Griesemer-Kupianinen-2015, Bach-Ballesteros-Koenenberg-Menrath-2017, Braunlich-Hasler-Lange-2018} and the references therein. 

In this paper we are concerned with the case of two photons which resembles the standard three-body situation. The Hilbert space of our system is $\CC^2\otimes\cF_s^2$, where $\cF_s^2$ is the truncated Fock space
\begin{equation}
\cF_s^2:=\CC \oplus L^2(\RR^d) \oplus L^2_s(\RR^d\times\RR^d)
\end{equation}
with $L^2_s(\RR^d\times\RR^d)$ standing for the subspace of $L^2(\RR^d\times\RR^d)$ consisting of symmetric functions, and equipped with 
the inner product 
\begin{equation}\label{inner.prod.in.Ls2}
(\phi,\psi)=\frac{1}{2}\int_{\RR^d}\int_{\RR^d}\phi(k_1,k_2)\ov{\psi(k_1,k_2)} \,\d k_1 \d k_2, \quad \phi,\psi\in L_s^2(\RR^d\times\RR^d).
\end{equation}
For $f=\bigl(f^{(\sigma)}_0, f^{(\sigma)}_1, f^{(\sigma)}_2\bigr)\in \CC^2\otimes\cF_s^2$, where $\sigma=\pm$ is the discrete variable, the Hamiltonian of our system is given by the formal expression 
\begin{equation}\label{Hamiltonian}
\begin{aligned}
(H_{\alpha}f)^{(\sigma)}_0 &= \sigma\eps f^{(\sigma)}_0 + \alpha\int_{\RR^d} \la(q)f^{(-\sigma)}_1\!(q) \,\d q,\\
(H_{\alpha}f)^{(\sigma)}_1\!(k) &=  (\sigma\eps+\omega(k))f^{(\sigma)}_1\!(k)+\alpha\la(k)f^{(-\sigma)}_0\! \\
&\hspace{1cm}+\alpha\int_{\RR^d}f^{(-\sigma)}_2\!(k,q)\la(q) \,\d q,\\
(H_{\alpha}f)^{(\sigma)}_2\!(k_1,k_2) &= (\sigma\eps+\omega(k_1)+\omega(k_2))f^{(\sigma)}_2\!(k_1,k_2)\\
&\hspace{1cm}+\alpha\la(k_1)f^{(-\sigma)}_1\!(k_2)+\alpha\la(k_2)f^{(-\sigma)}_1\!(k_1). 
\end{aligned}
\end{equation}
Here $\pm\eps$ ($\eps>0$) are the energy levels of the atom corresponding to its excited and ground states, respectively, $\omega(k)=|k|$ is the photon dispersion relation, $\alpha>0$ is the coupling constant and $\la$ is the coupling function given by the product of $\sqrt{\omega(k)}$ with a cut-off function for large~$k$. The spatial dimension, $d\geq1$, plays no particular role in our analysis and is left arbitrary. 

In general, the dispersion relation $\omega\geq0$ and the coupling function $\la$ are fixed by the physics of the problem. Motivated by different applications of \eqref{Hamiltonian} (e.g. in solid state physics, see \cite{Spohn-CMP-1989}), 
one likes to consider them as free parameter functions and impose only some general conditions such as 
\begin{equation}\label{cond:la.la.omega.in.L2}
\la\in L^2(\RR^d), \quad \frac{\la}{\sqrt{\omega}}\in L^2(\RR^d)
\end{equation}
whenever $\inf\omega=0$. To the best of our knowledge, every study on the spectrum of the spin-boson model in the up-to-date literature assumes at least \eqref{cond:la.la.omega.in.L2} or its strengthened version where the second condition in \eqref{cond:la.la.omega.in.L2} is replaced by  
\begin{equation}\label{cond2:la.omega.in.L2}
\frac{\la}{\omega}\in L^2(\RR^d),
\end{equation}
which is known as the \emph{infrared regularity condition} (see, for example, \cite{Hirokawa-2001RMP, Davies1981_SymmetryBreaking}).

There are many papers in the current literature containing rigorous results on the spectrum of $H_{\alpha}$ for \emph{small} coupling. It is known that, under appropriate conditions in addition to \eqref{cond:la.la.omega.in.L2}-\eqref{cond2:la.omega.in.L2}, there exists a sufficiently small coupling constant $\alpha_0>0$ such that for all $\alpha\in (0,\alpha_0)$ the spectrum of $H_{\alpha}$ consists of a unique simple discrete eigenvalue separated by a gap from the purely absolutely continuous part of the spectrum which is a half-line. This result was obtained in \cite{Minlos-Spohn-1996} by means of the scattering theory and also in \cite{Huebner-Spohn-1995} by means of the conjugate operator method (here we would like to note that \cite{Huebner-Spohn-1995} as well as \cite{Huebner94atominteracting, Huebner-Spohn-1995-review} treat also the case of arbitrarily finite number of photons). 

On the other hand, not much is known for general coupling, without any assumptions about its smallness. There are rather few papers devoted to the spectral theory of the spin-boson model (with or without particle number cut-off) as well as more general abstract models sometimes called Pauli-Fierz models or generalized spin-boson models; see, for instance, \cite{Huebner-Spohn-1995-review, Huebner94atominteracting, Arai-Hirokawa-1997, Derezinski-Gerard-Rev.Amth.Phys-1999, Georgescu-Gerard-Moller-CMP2004, Hasler-Herbst-2011}. In these studies the location of the essential spectrum is given for any value of the coupling constant, and results including finiteness of the point spectrum or absence of the singular continuous spectrum are proven under various assumptions in addition to \eqref{cond:la.la.omega.in.L2}. We would also like to mention \cite{Merkli-CMP2015} where the authors study the dynamics of the spin-boson model at arbitrary coupling strength. There is a recent related work \cite{MNR-2016-1D} on the location of the bottom of the essential spectrum and the finiteness of the number of eigenvalues on the left of it. Although \cite{MNR-2016-1D} does not need assumptions such as \eqref{cond2:la.omega.in.L2}, it is required that the underlying domain is a compact subset of the real line and that the parameter functions $\la$ and $\omega$ are real-analytic. 

In the present paper we aim to establish the finiteness of the discrete spectrum along with an explicit description of the essential spectrum under fairly weak regularity conditions on the parameter functions $\omega$ and $\la$ by developing a new, simple and instructive approach. It turns out that the continuity of $\omega$ and the condition $\la\in L^2(\RR^d)$ are sufficient to achieve our goal. In particular, neither the infrared regularity \eqref{cond2:la.omega.in.L2} nor the second condition in \eqref{cond:la.la.omega.in.L2} is needed. We would also like to emphasize that our main result -- Theorem~\ref{thm1} -- holds even if the level set of $\omega$ corresponding to the photon mass is of positive Lebesgue measure, whereas it was always assumed to be a finite or a Lebesgue null set in the up-to-date literature. The methods we employ to achieve these results are direct and of abstract nature, allowing for simpler proofs. In particular, we benefit from block operator matrix techniques involving Schur complements and the corresponding Frobenius-Schur factorizations combined with the standard perturbation theory. 

The paper is organized as follows. In Section~\ref{sec:main.results} we give precise formulations of our main results (Theorems~\ref{thm1}-\ref{thm2}). In Section~\ref{subsec:prelimiaries} we explain the reduction of the problem to the spectral analysis of a $2\times 2$ operator matrix and describe the Schur complement of the latter. The detailed proofs of the main results are given in Sections~\ref{subsec:proof.of.thm1}-\ref{subsec:proof.of.thm2}. 

Throughout the paper we adopt the following notations. For a self-adjoint operator $T$ acting in a Hilbert space and a constant $\mu \in \RR$ such that $\mu\leq\min\sess(T)$, we denote by $N(\mu; T)$ the number of the eigenvalues of $T$ less than $\mu$ (counted with multiplicities). Note that $N(\mu; T)$ coincides with the dimension of the spectral subspace of $T$ corresponding to the interval $(-\infty,\mu)$, 
see \cite[Section~IX]{Birman-Solomjak-87b}. The integrals with no indication of the limits imply the integration over the whole
space $\RR^d$ or $\RR^d\times\RR^d$. The Euclidean and operator norms are denoted by $|\mydot|$ and $\|\mydot\|$, respectively. 

\section{Summary of the main results}\label{sec:main.results}

Throughout the paper we assume the following hypotheses. 
\begin{assumption}\label{assumption}
	\emph{The parameter $\eps>0$ is fixed, the photon dispersion relation \linebreak $\omega:\RR^d\to\RR$ is an unbounded continuous function with 
		\begin{equation}\label{def:m}
		m:=\inf_{k\in\RR^d} \omega(k)\geq0, 
		\end{equation}
		the coupling function $\la:\RR^d\to\CC$ is not identically zero and 
		satisfies}
	\begin{equation}\label{ass:la.in.L2}
	\la\in L^2(\RR^d).
	\end{equation}
\end{assumption}
\begin{remark}
	If $\la\equiv0$ 
	on $\RR^d$, then the photons do not couple to the atom and the description of the spectrum is straightforward. As the relativistic photon dispersion relation $\omega(k)=\sqrt{k^2+m^2}$ suggests, the constant $m$ defined in \eqref{def:m} may play a role of the ``mass" of the photon.
\end{remark} 
The natural domain of the unperturbed operator $H_0$ is given by
\begin{equation}\label{nat.dom.H0}
\Dom(H_0):=\CC^2\otimes\{\CC\oplus\cH_1\oplus\cH_2\}
\end{equation}
with $\cH_1$ and $\cH_2$ standing for the weighted $L^2$- Hilbert spaces
\begin{equation}\label{Dom:H1}
\cH_1:=\Big\{f\in L^2(\RR^d): \int|\omega(k)|^2|f(k)|^2\,\d k<\infty\Big\}
\end{equation}
and
\begin{equation}\label{Dom:H2}
\cH_2:=\Big\{g\in L^2_s(\RR^d\times\RR^d): \int|\omega(k_1)+\omega(k_2)|^2|g(k_1,k_2)|^2\,\d k_1\d k_2<\infty\Big\}.
\end{equation}
The condition \eqref{ass:la.in.L2} implies the boundedness of the perturbation and thus the expression for $H_{\alpha}$ given in \eqref{Hamiltonian} generates a self-adjoint operator in the Hilbert space $\CC^2\otimes\cF_s^2$ on the natural domain of $H_0$ (see \cite[Theorem~V.4.3]{Kato}). For notational convenience, we denote the corresponding self-adjoint operator again by~$H_{\alpha}$.

Consider the continuous functions
\begin{equation}\label{Phi}
\Phi^{(\sigma)}_{\alpha}(z)=-\sigma\eps-z-\alpha^2\int\frac{|\la(q)|^2 \,\d q}{\omega(q)+\sigma\eps-z}, \quad z\in(-\infty,m+\sigma\eps)
\end{equation}
with the discrete variable $\sigma=\pm$ and the constant $\eps>0$ corresponding to the excited state energy of the atom. It is easy to see that $\Phi^{(\sigma)}_{\alpha}$ are strictly decreasing 
and 
\begin{equation}\label{lim.at.-infty}
\lim_{z\downarrow-\infty}\Phi^{(\sigma)}_{\alpha}(z)=+\infty.
\end{equation}
Moreover, the monotone convergence theorem implies the existence of the (possibly improper) limits
\begin{equation}\label{mon.conv.thm:lim.at.boundary}
\lim_{z\uparrow m+\sigma\eps}\Phi^{(\sigma)}_{\alpha}(z)=:\Phi^{(\sigma)}_{\alpha}(m+\sigma\eps).
\end{equation}
We distinguish the two cases:
\begin{enumerate}[\underline{\upshape{Case}} 1:]
	\item \emph{It holds that}
	\begin{equation}\label{strong.ultrared.singularity}
	\frac{\la}{\sqrt{\omega-m}}\notin L^2(\RR^d).
	\end{equation} 
	In this case the limits in \eqref{mon.conv.thm:lim.at.boundary} are negative infinity for all $\alpha>0$. Hence, \eqref{lim.at.-infty} and the monotonicity imply that each of the continuous functions $\Phi^{(\sigma)}_{\alpha}$ has \emph{a unique zero}, denoted by $E_{\sigma\eps}(\alpha)$,  in the interval $(-\infty,m+\sigma\eps)$ for all $\alpha>0$. We define
	\begin{equation}\label{def1:E(alpha)}
	E(\alpha):=\min\{E_{\eps}(\alpha), E_{-\eps}(\alpha)\}. 	
	\end{equation}
	\item \emph{It holds that}
	\begin{equation}\label{weak.ultrared.regularity}
	\frac{\la}{\sqrt{\omega-m}}\in L^2(\RR^d).
	\end{equation} 
	In this case the limits in \eqref{mon.conv.thm:lim.at.boundary} are finite for all $\alpha>0$ and given by
	\begin{equation}
	\Phi^{(\sigma)}_{\alpha}(m+\sigma\eps)=-2\sigma\eps-m-\alpha^2\int\frac{|\la(q)|^2 \d q}{\omega(q)-m}.
	\end{equation}
	We distinguish the two subcases: 
	\begin{enumerate}[\upshape a)]
		\item \emph{Either $m\geq2\eps$ and $\alpha>0$ is arbitrary, or $m<2\eps$ and $\alpha>\alpha_{\emph{cr}}$, where} 
		\begin{equation}\label{alpha0}
		\ds\alpha_{\text{cr}}:=\frac{\sqrt{2\eps-m}}{\bigl\|\frac{\la}{\sqrt{\omega-m}}\bigr\|_{L^2(\RR^d)}}.
		\end{equation}
		In this case $\Phi^{(\sigma)}_{\alpha}(m+\sigma\eps)<0$ for each $\sigma=\pm$. Hence, \eqref{lim.at.-infty} and the monotonicity again imply that each of the continuous functions $\Phi^{(\sigma)}_{\alpha}$ has \emph{a unique zero} (again denoted by) $E_{\sigma\eps}(\alpha)$ in the interval $(-\infty,m+\sigma\eps)$. Once again, we define $E(\alpha)$ as in \eqref{def1:E(alpha)}.
		\item \emph{The case $m<2\eps$ and $0<\alpha\leq\alpha_{\emph{cr}}$.} In this case $\Phi^{(-)}_{\alpha}(m-\eps)\geq0$ and $\Phi^{(+)}_{\alpha}(m+\eps)<0$. Consequently, \eqref{lim.at.-infty} and the monotonicity of the continuous functions $\Phi^{(\sigma)}_{\alpha}$ imply that $\Phi^{(-)}_{\alpha}$ does not vanish in the interval $(-\infty,m-\eps)$, whereas $\Phi^{(+)}_{\alpha}$ has \emph{a unique zero}, denoted by $E_{\eps}(\alpha)$, in the interval $(-\infty,m+\eps)$. In this case, we define
		\begin{equation}\label{def2:E(alpha)}
		E(\alpha):=E_{\eps}(\alpha).
		\end{equation}
	\end{enumerate}
\end{enumerate}

The following theorems summarize the main results of the present paper.

\begin{theorem}\label{thm1}
	Let Assumption~\ref{assumption} hold. Let the coupling constant $\alpha>0$ be arbitrary and let $E(\alpha)$ be defined as in one of \eqref{def1:E(alpha)} and \eqref{def2:E(alpha)}. Then  
	\begin{enumerate}[\upshape i)]
		\item
		\label{thm:part.i} 
		the essential spectrum of $H_\alpha$ is given by
		\begin{equation}\label{form.ess.spec.H}
		\sess(H_{\alpha})=[m+E(\alpha),\infty). 
		\end{equation}
		\item\label{thm:part.ii} 
		the discrete spectrum of $H_{\alpha}$ is finite.
	\end{enumerate}		 
\end{theorem}

The next result provides more explicit description of the bottom of the essential spectrum for weak coupling.
\begin{theorem}\label{thm2}
	Let Assumption~\ref{assumption} hold. Assume that either \eqref{weak.ultrared.regularity} is satisfied or $m>0$. Then, for all sufficiently small $\alpha>0$, we have  
	\begin{equation}\label{form.ess.spec.H.weak.coupling}
	\sess(H_{\alpha})=[m+E_{\eps}(\alpha),\infty). 
	\end{equation}
	Moreover, the asymptotic expansion holds
	\begin{equation}\label{asymp.exp.bottom}
	E_{\eps}(\alpha)=-\eps-\alpha^2\Bigl\|\frac{\la}{\sqrt{\omega+2\eps}}\Bigr\|^2_{L^2(\RR^d)} + \rm{o}(\alpha^2), \quad \alpha\downarrow 0.
	\end{equation}
\end{theorem} 

\begin{remark}
	\begin{enumerate}[\upshape i)]
		\item Whenever \eqref{weak.ultrared.regularity} is satisfied, one can make it more explicit how small should the coupling constant be so that $E(\alpha)=E_\eps(\alpha)$ 
		holds. In fact, it follows from the proof of Theorem~\ref{thm2} that  $E(\alpha)=E_\eps(\alpha)$ for all $\alpha>0$ satisfying
		\begin{equation}\label{bound.on.small.alpha}
		\ds\alpha\leq\frac{\sqrt{2\eps}}{\bigl\|\frac{\la}{\sqrt{\omega-m}}\bigr\|_{L^2(\RR^d)}}.
		\end{equation}
		In the massless case \eqref{bound.on.small.alpha} is equivalent to $\alpha\leq\alpha_{\text{cr}}$ and $E(\alpha)=E_\eps(\alpha)$ holds trivially (see \eqref{def2:E(alpha)}).
		\item There is no difficulty in verifying that $E(\alpha)$ corresponds to the ground state energy of the one-boson Hamiltonian (see Section~\ref{subsec:appendix.2}). In the massless case $m=0$ Theorem~\ref{thm1}\ref{thm:part.i}) is thus an analogue of the Hunziker--van Winter--Zhislin (HVZ) theorem of the standard three-body Schr\"odinger operators, see~\cite[Section~XIII]{Reed-Simon-IV} and \cite{Huebner-Spohn-1995-review}. 
	\end{enumerate}
\end{remark}	

\section{Proofs of the main results}\label{sec:proofs}
Unless specified, we assume throughout this section that the discrete variable $\sigma=\pm$ is fixed.
\subsection{Preliminaries}\label{subsec:prelimiaries}
By means of the unitary transformation $U:\CC^2\otimes\cF_s^2\to\cF_s^2\oplus\cF_s^2$, defined by
\begin{equation}
U:
\begin{pmatrix}
\begin{pmatrix}
f_0^{(+)}\\
f_0^{(-)}
\end{pmatrix},
\begin{pmatrix}
f_1^{(+)}\\
f_1^{(-)}
\end{pmatrix}, 
\begin{pmatrix}
f_2^{(+)}\\
f_2^{(-)}
\end{pmatrix}
\end{pmatrix}
\mapsto
\begin{pmatrix}
\begin{pmatrix}
f_0^{(+)}\\
f_1^{(-)}\\
f_2^{(+)}
\end{pmatrix},
\begin{pmatrix}
f_0^{(-)}\\
f_1^{(+)}\\
f_2^{(-)}
\end{pmatrix}
\end{pmatrix},
\end{equation}
we can block-diagonalize the Hamiltonian $H_{\alpha}$ in \eqref{Hamiltonian}, that is, 
\begin{equation}\label{diagonalization.of.Hamiltonian}
U^*H_{\alpha}U={\rm{diag}}\{H_{\alpha}^{(+)}, H_{\alpha}^{(-)}\}
\end{equation}
where 
\begin{equation}
H_{\alpha}^{(\sigma)}:=
\begin{pmatrix}
H^{(\sigma)}_{00} & \alpha H_{01} & 0\\[0.5ex]
\alpha H_{10} & H^{(\sigma)}_{11} & \alpha H_{12}\\[0.5ex]
0 & \alpha H_{21} & H^{(\sigma)}_{22}
\end{pmatrix}
\end{equation}
is the tridiagonal operator matrix acting in the truncated Fock space $\cF_s^2$ on the~domain 
\begin{equation}\label{domain.of.H.sigma}
\Dom(H_{\alpha}^{(\sigma)}):=\CC\oplus\cH_1\oplus\cH_2
\end{equation} 
with the weighted Hilbert spaces $\cH_1$, $\cH_2$ given in \eqref{Dom:H1}-\eqref{Dom:H2}. The operator entries of $H_{\alpha}^{(\sigma)}$ are defined by 
\begin{equation}
\begin{aligned}
&H^{(\sigma)}_{00} f_0 = \sigma\eps f_0, \quad  H_{01}f_1 = \int \la(q)f_1(q) \,\d q, \quad (H_{10}f_0)(k)=f_0\ov{\la(k)},\\ \label{def.op.entries.2}
&(H^{(\sigma)}_{11} f_1)(k) = (-\sigma\eps+\omega(k))f_1(k), \quad (H_{12}f_2)(k)= \int f_2(k,q)\la(q) \,\d q,\\ 
&(H_{21}f)(k_1,k_2) = \ov{\la(k_1)}f(k_2)+\ov{\la(k_2)}f(k_1)
\end{aligned}
\end{equation}
and 
\begin{equation*}
(H^{(\sigma)}_{22}f_2)(k_1,k_2) = (\sigma\eps+\omega(k_1)+\omega(k_2))f_2(k_1,k_2), \quad (f_0,f_1,f_2)\in\cF_s^2.
\end{equation*}
It follows from the condition \eqref{ass:la.in.L2} that $H_{12}\!:\!L^2_s(\RR^d\times\RR^d)\!\to\! L^2(\RR^d)$ and $H_{01}\!:\!L^2(\RR^d)\!\to\!\CC$ are bounded operators with $H_{10}=H_{01}^*$ and $H_{21}=H_{12}^*$. Hence, $H_{\alpha}^{(\sigma)}$ is self-adjoint on the domain \eqref{domain.of.H.sigma} (see \cite[Theorem~V.4.3]{Kato}). 

Denoting by $P:\cF_s^2\to\cH_1\oplus\cH_2$ the projection operator onto the last two components in the Hilbert space $\cF_s^2$, we can decompose $H_{\alpha}^{(\sigma)}$ as follows
\begin{equation}\label{splitting.of.H.sigma}
H_{\alpha}^{(\sigma)}=P^*\wh H_{\alpha}^{(\sigma)}P + 
\begin{pmatrix}
H^{(\sigma)}_{00} & \alpha H_{01} & 0\\[1ex]
\alpha H^*_{01} & 0 & 0\\[1ex]
0 & 0 & 0			
\end{pmatrix},
\end{equation}
where 
\begin{equation}\label{op.matrix.A}
\wh H_{\alpha}^{(\sigma)}:=\begin{pmatrix}
H^{(\sigma)}_{11} & \alpha H_{12}\\[1ex]
\alpha H^*_{12} & H^{(\sigma)}_{22}
\end{pmatrix}, \quad \Dom(\wh H_{\alpha}^{(\sigma)})=\cH_1\oplus\cH_2.
\end{equation}
The second operator on the right-hand-side of \eqref{splitting.of.H.sigma} maps $\cF_s^2$ onto the two-dimensional subspac $\CC\oplus\text{Span}\{\la\}\oplus\{0\}\subset\cF_s^2$. Since the essential spectrum as well as the finiteness of the discrete spectrum of self-adjoint operators are invariant with respect to finite-rank perturbations (see, for example, \cite[Chapter~9]{Birman-Solomjak-87b}), in view of \eqref{diagonalization.of.Hamiltonian}, it follows that 
\begin{equation}\label{ess.spec.relation}
\sess(H_{\alpha})=\sess(\wh H_{\alpha}^{(+)}) \cup \sess(\wh H_{\alpha}^{(-)})
\end{equation}
and 
\begin{equation}\label{disc.spec.relation}
N\bigl(z; H_{\alpha}\bigr) \leq N\bigl(z; \wh H_{\alpha}^{(+)}\bigr) + N\bigl(z; \wh H_{\alpha}^{(-)}\bigr)+4
\end{equation}
for all $z\leq\min\sess(H_{\alpha})$ provided that the right-hand-side of \eqref{disc.spec.relation} is finite. 
That is why, from now on we completely focus on the $2\times 2$ operator matrix $\wh H_{\alpha}^{(\sigma)}$ defined in \eqref{op.matrix.A} and acting in the Hilbert space
\begin{equation}
\cH:=L^2(\RR^d)\oplus L^2_s(\RR^d\times\RR^d). 
\end{equation} 
It is easy to see that $H^{(\sigma)}_{22}:L^2_s(\RR^d\times\RR^d)\to L^2_s(\RR^d\times\RR^d)$ is a self-adjoint operator on the domain $\Dom(H^{(\sigma)}_{22})=\cH_2$ with the spectrum
\begin{equation}
\sigma(H^{(\sigma)}_{22})=[2m+\sigma\eps,\infty).
\end{equation}

As it was mentioned in Introduction, our approach is based on the so-called Schur complement and the corresponding Frobenius-Schur factorization. This has proven to be a powerful and natural tool when dealing with spectral properties of $2\times 2$ operator matrices such as $\wh H_{\alpha}^{(\sigma)}$, see e.g. \cite{Tre08}. We would like to mention that the Schur complement method employed in this paper was used previously in the spectral analysis of Pauli-Fierz Hamiltonians – under the name Feshbach method or Feshbach-Schur method – in \cite{Bach-Froechlich-Sigal-Adv.Math1998} and \cite{DerJak-JFA2001}, for instance.

For the rest of this subsection we assume $z\in(-\infty, 2m+\sigma\eps)=\rho(H^{(\sigma)}_{22})$ to be fixed. The Schur complement of $H^{(\sigma)}_{22}-z$ in the operator matrix $\wh H_{\alpha}^{(\sigma)}-z$ is defined by  
\begin{equation*}
S_{\alpha}^{(\sigma)}(z) := H^{(\sigma)}_{11}-z-\alpha^2H_{12}(H^{(\sigma)}_{22}-z)^{-1}H_{12}^*, \quad \Dom(S_{\alpha}^{(\sigma)}(z)) := \cH_1.
\end{equation*}
It can be checked easily that 
\begin{equation}\label{Schur2}
S_{\alpha}^{(\sigma)}(z) = \Delta_{\alpha}^{(\sigma)}(z)-\alpha^2 K^{(\sigma)}(z),
\end{equation}
where, $\Delta_{\alpha}^{(\sigma)}(z):L^2(\RR^d)\to L^2(\RR^d)$ is the operator of multiplication by the function 
\begin{equation}\label{def:Delta}
\begin{aligned}
\Delta_{\alpha}^{(\sigma)}(k;z):\!&= \Phi_{\alpha}^{(\sigma)}(z-\omega(k))\\
&=\omega(k)-\sigma\eps-z-\alpha^2\int\frac{|\la(q)|^2 \d q}{\omega(k)+\omega(q)+\sigma\eps-z}
\end{aligned}
\end{equation}
on the domain
\begin{equation}
\Dom(\Delta_{\alpha}^{(\sigma)}(z)):=\cH_1
\end{equation}
and $K^{(\sigma)}(z): L^2(\RR^d)\to L^2(\RR^d)$ is the integral operator with the kernel 
\begin{equation}\label{kernel:p}
p^{(\sigma)}(k_1,k_2;z):=\frac{\ov{\la(k_1)}\la(k_2)}{\omega(k_1)+\omega(k_2)+\sigma\eps-z}
\end{equation}
on the domain
\begin{equation}
\Dom(K^{(\sigma)}(z)):=\cH_1.
\end{equation} 
Since $F^{(\sigma)}(z):=H_{12}(H^{(\sigma)}_{22}-z)^{-1}$ is a bounded operator from $L^2_s(\RR^d\times\RR^d)$ to $L^2(\RR^d)$, the following Frobenius-Schur factorization holds 
\begin{equation}\label{Frobenius-Schur}
\wh H_{\alpha}^{(\sigma)}-z = V_{\alpha}^{(\sigma)}(z) W_{\alpha}^{(\sigma)}(z)\bigl(V_{\alpha}^{(\sigma)}(z)\bigr)^*
\end{equation}
where
\begin{equation}
V_{\alpha}^{(\sigma)}(z):=
\begin{pmatrix}
I & \! \alpha F^{(\sigma)}(z)\\[1ex]
0 & \! I
\end{pmatrix}, \quad
W^{(\sigma)}(z):= 
\begin{pmatrix}
S_{\alpha}^{(\sigma)}(z) & \! 0\\[1ex]
0 & \! H^{(\sigma)}_{22}-z
\end{pmatrix}
\end{equation}
see \cite[Chapter~2]{Tre08}. One of the important consequences of this factorization is the relation
\begin{equation}\label{Schur.power.ess.spec}
\sess(\wh H_{\alpha}^{(\sigma)})\cap\rho(H^{(\sigma)}_{22})=\sess(S_{\alpha}^{(\sigma)})\cap\rho(H^{(\sigma)}_{22}).
\end{equation}
Moreover, we have the following result.
\begin{lemma}\label{lem:EV.count.1}
	Let $z<2m+\sigma\eps$. Then 
	\begin{equation}
	N\bigl(z;\wh H_{\alpha}^{(\sigma)}\bigr) = N\bigl(0;S_{\alpha}^{(\sigma)}(z)\bigr).
	\end{equation}
\end{lemma}
\begin{proof}
	The map $V_{\alpha}^{(\sigma)}(z):\cH\to\cH$ is bounded and bijective. Since $z<\min\sigma(H^{(\sigma)}_{22})$, it is obvious that    
	\begin{equation}\label{D.has.no.EV}
	N\bigl(0;H^{(\sigma)}_{22}-z\bigr)=0.
	\end{equation}
	The Frobenius-Schur factorization \eqref{Frobenius-Schur} yields 
	\begin{equation}\label{lem:estimates.EV.cf.100}
	\begin{aligned}
	N\bigl(z;\wh H_{\alpha}^{(\sigma)}\bigr) &=N\bigl(0;\bigl(V_{\alpha}^{(\sigma)}(z)\bigr)^{-1}(\wh H_{\alpha}^{(\sigma)}-z)\bigl(\bigl(V_{\alpha}^{(\sigma)}(z)\bigr)^{-1}\bigr)^*\bigr)\\
	&=N\bigl(0;W_{\alpha}^{(\sigma)}(z)\bigr) =N\bigl(0;S_{\alpha}^{(\sigma)}(z)\bigr), 
	\end{aligned} 
	\end{equation}
	where the last equality is an immediate consequence of \eqref{D.has.no.EV}.
\end{proof}

\begin{remark}\label{rem:on.Delta.and.K}
	In view of Assumption~\ref{assumption} it follows that, for each $z<2m+\sigma\eps$, 
	\begin{enumerate}[\upshape i)]
		\item the operators $S_{\alpha}^{(\sigma)}(z)\!:\!L^2(\RR^d)\to L^2(\RR^d)$ and $\Delta_{\alpha}^{(\sigma)}(z)\!:\!L^2(\RR^d)\to L^2(\RR^d)$ are bounded from below and self-adjoint on $\cH_1$;
		\item the integral operator $K^{(\sigma)}(z):L^2(\RR^d)\to L^2(\RR^d)$ is self-adjoint and Hilbert-Schmidt as its kernel $p^{(\sigma)}(\mydot,\mydot;z)$	belongs to $L^2(\RR^d\!\times\RR^d)$ with 
		\begin{equation}
		\|p^{(\sigma)}(\mydot,\mydot;z)\|_{L^2(\RR^d\times\RR^d)}\leq  \frac{\|\la\|^2_{L^2(\RR^d)}}{2m+\sigma\eps-z}.
		\end{equation} 
	\end{enumerate} 
\end{remark}

\subsection{Proof of Theorem~\ref{thm1}}\label{subsec:proof.of.thm1}
We recall that the function $\Phi^{(\sigma)}_{\alpha}$ defined in \eqref{Phi} is strictly decreasing in the interval $(-\infty, m+\sigma\eps)$. Hence, 
\begin{equation}\label{relation.between.P.and.Delta}
\inf_{k\in\RR^d}\Delta^{(\sigma)}_{\alpha}(k;z)=\inf_{k\in\RR^d}\Phi^{(\sigma)}_{\alpha}(z-\omega(k))=\Phi^{(\sigma)}_{\alpha}(z-m), \quad z<2m+\sigma\eps.
\end{equation}
Conceptually, the proof of the first part of Theorem~\ref{thm1} is similar to that of \cite[Theorem~3.1]{Ibrogimov-Tretter-OTAA2018}, but the technical details are quite different.	

\subsubsection{Proof of \upshape{Theorem~\ref{thm1}\ref{thm:part.i})}} 

The proof will be done in three steps.

\smallskip
\noindent
\underline{Step} 1: In this step we show that
\begin{equation}
[2m+\sigma\eps,\infty)\subset\sess(\wh H_{\alpha}^{(\sigma)}).
\end{equation}
Let $z_0\in(2m+\sigma\eps,\infty)$ be arbitrary. Then $z_0=2\omega(k_0)+\sigma\eps$ for some $k_0\in\RR^d$. Further, let $\varphi\in C^\infty(\RR^d)$ be an arbitrary function supported in the annulus $\{k\in \RR^d : 0.5\leq\|k\|<1\}$ and such that $\|\varphi\|_{L^2(\RR^d)}=1$. Then the sequence $\{\varphi_n\}_{n \in \NN}$ defined by  
\begin{equation}
\varphi_n(k) := 2^{\frac{nd}{2}} \varphi(2^n(k-k_0)), \quad k \in \RR^d, 
\end{equation}
is an orthonormal system in $L^2(\RR^d)$. It follows that the sequence $\{\psi_n\}_{n\in\NN}$ defined by
\begin{equation}
\psi_n(k_1,k_2)=\sqrt{2}\varphi_n(k_1) \varphi_n(k_2), \quad k_1,k_2\in\RR^d,
\end{equation} 
is an orthonormal system in $L^2_s(\RR^d\times\RR^d)$ and the Lebesgue's dominated convergence theorem implies that  
\begin{equation}\label{D.sing.seq}
\|(H^{(\sigma)}_{22}-z_0) \psi_n\|_{L^2_s(\RR^d\times\RR^d)} \to 0, \quad \text{as} \quad n\to\infty.
\end{equation}
On the other hand, in view of \eqref{ass:la.in.L2}, the weak convergence (to zero) of the orthonormal system $\{\varphi_n\}_{n\in\NN}$ in $L^2(\RR^d)$ yields
\begin{equation}\label{sing.seq.B}
\|H_{12}\psi_n\|_{L^2(\RR^d)} \leq \sqrt{2}\left|\int \la(q) \varphi_n(q)\,\d q \right|\to 0, \quad n \to \infty.
\end{equation}
It is obvious that the sequence $\{\Psi_n\}_{n\in\NN}:=\{(0,\psi_n)^t\}_{n\in\NN}$ is an orthonormal system in $\cH=L^2(\RR^d)\oplus L^2_s(\RR^d\times\RR^d)$. Hence, the relation  
\begin{equation*}
\|(\wh H_{\alpha}^{(\sigma)}-z_0) \Psi_n\|^2_{\cH}= \alpha^2\|H_{12} \psi_n\|^2_{L^2(\RR^d)}+ \|(H^{(\sigma)}_{22}-z_0) \psi_n\|^2_{L^2_s(\RR^d\times\RR^d)}
\end{equation*}
together with \eqref{D.sing.seq}-\eqref{sing.seq.B} imply that $\{\Psi_n\}_{n \in \NN}$ is a singular sequence for $\wh H_{\alpha}^{(\sigma)}-z_0$. That is why $z_0\in\sess(\wh H_{\alpha}^{(\sigma)})$, see \cite[Theorem~IX.1.2]{Birman-Solomjak-87b}. Because $z_0\in(2m+\sigma\eps,\infty)$ was arbitrary, the claim follows from the closedness of the essential spectrum.

\medskip
\noindent
\underline{Step} 2: In this step we prove that the set $\sess(\wh H_{\alpha}^{(\sigma)}) \cap (-\infty,2m+\sigma\eps)$ is empty if \eqref{weak.ultrared.regularity} holds together with $m<-2\sigma\eps$ and $0<\alpha\leq\alpha_{\text{cr}}$, and consists of $[m+E_{\sigma\eps}(\alpha),2m+\sigma\eps)$, otherwise. To this end, let $z\in(-\infty,2m+\sigma\eps)$ be fixed. Recalling the relation \eqref{Schur.power.ess.spec}, we obtain
\begin{equation}\label{ess.equiv1}
z \in \sess(\wh H_{\alpha}^{(\sigma)}) \quad \Longleftrightarrow \quad 0 \in \sess(S_{\alpha}^{(\sigma)}(z)).
\end{equation}
Because $K^{(\sigma)}(z):L^2(\RR^d)\to L^2(\RR^d)$ is compact operator and the essential spectrum is invariant with respect to compact perturbations, we have     
\begin{equation}\label{ess.spec.S(z)}
\sess(S_{\alpha}^{(\sigma)}(z)) = \sess(\Delta_{\alpha}^{(\sigma)}(z)-K^{(\sigma)}(z))=\sess(\Delta_{\alpha}^{(\sigma)}(z)),
\end{equation}
see \eqref{Schur2}, Remark~\ref{rem:on.Delta.and.K} and \cite[Theorem~IX.2.1]{EE87}. On the other hand, $\Delta_{\alpha}^{(\sigma)}(z)$ is the operator of multiplication by the continuous function $\Delta_{\alpha}^{(\sigma)}(\mydot;z)$ and thus $\sess(\Delta_{\alpha}^{(\sigma)}(z))$ coincides with the closure of the range of the function $\Delta_{\alpha}^{(\sigma)}(\mydot;z)$ over $\RR^d$. Hence, we infer from \eqref{ess.equiv1}-\eqref{ess.spec.S(z)} that $z \in \sess(\wh H_{\alpha}^{(\sigma)})$ if and only if $0$ is a limit point of the range of the function $\Delta_{\alpha}^{(\sigma)}(\mydot;z)$ over $\RR^d$. Since $\omega$ is unbounded, the monotonicity of the function $\Phi_{\alpha}^{(\sigma)}$ implies that
\begin{equation}\label{sup.omega.infty}
\sup_{k\in\RR^d}\Delta_{\alpha}^{(\sigma)}(k;z)=\sup_{k\in\RR^d}\Phi_{\alpha}^{(\sigma)}(z-\omega(k))=+\infty.
\end{equation}
Recalling \eqref{relation.between.P.and.Delta} and the continuity of $\Delta_{\alpha}^{(\sigma)}(\mydot;z)$ in $z$, we thus obtain
\begin{equation}\label{ess.spec.abstract.1}
\sess(\wh H_{\alpha}^{(\sigma)}) \cap (-\infty,2m+\sigma\eps)=\bigl\{z<2m+\sigma\eps: \Phi^{(\sigma)}_{\alpha}(z-m)\leq 0\bigr\}.
\end{equation}
On the other hand, the function $\Phi^{(\sigma)}_{\alpha}(\mydot-m)$ is strictly decreasing in $(-\infty,2m+\sigma\eps)$ and thus $\Phi^{(\sigma)}_{\alpha}(z-m)>\Phi^{(\sigma)}_{\alpha}(m+\sigma\eps)\geq0$ for each $z<2m+\sigma\eps$ whenever \eqref{weak.ultrared.regularity} holds together with $m<-2\sigma\eps$ and $0<\alpha\leq\alpha_{\text{cr}}$. Otherwise, $\Phi^{(\sigma)}_{\alpha}(m+\sigma\eps)<0$ and $\Phi^{(\sigma)}_{\alpha}(z-m)<\Phi^{(\sigma)}_{\alpha}(E_{\sigma\eps}(\alpha))=0$ for all $z\in (m+E_{\sigma\eps}(\alpha), 2m+\sigma\eps)$. However, $\Phi^{(\sigma)}_{\alpha}(z-m)>\Phi^{(\sigma)}_{\alpha}(E_{\sigma\eps}(\alpha))=0$ for all $z<m+E_{\sigma\eps}(\alpha)$. In view of \eqref{ess.spec.abstract.1}, these observations yield the desired result.

\medskip
\noindent
\underline{Step} 3: In this step we combine previous two steps to derive \eqref{form.ess.spec.H}. To this end, first we recall that $m\geq0$ and $\eps>0$. Hence, the two steps yield  
\begin{equation}\label{sess.H+}
\sess(\wh H_{\alpha}^{(+)})=[m\!+\!E_{\eps}(\alpha),\infty) \quad \text{for all} \quad \alpha>0,
\end{equation}
whereas 
\begin{equation}\label{sess.H-}
\sess(\wh H_{\alpha}^{(-)})\!=\!\begin{cases} 
[2m\!-\!\eps,\infty) &\!\!\!\! \mbox{if \eqref{weak.ultrared.regularity} holds, } m\!<\!2\eps \mbox{ and } \alpha\leq\alpha_{\text{cr}}, \\ 
[m\!+\!E_{-\eps}(\alpha), \infty) &\!\!\!\! \mbox{otherwise}.
\end{cases}
\end{equation}
It is easy to see that $m-\eps>E_\eps(\alpha)$ for all $\alpha>0$, which is a simple consequence of the relation
\begin{equation}\label{m-e.big.E(alpha)}
\Phi^{(+)}_{\alpha}(m-\eps)=-m-\alpha^2\int\frac{|\la(q)|^2 \d q}{\omega(q)-m+2\eps}<0.
\end{equation}
Therefore, combining \eqref{sess.H+}, \eqref{sess.H-} and \eqref{ess.spec.relation}, we obtain \eqref{form.ess.spec.H}.
\qed

\subsubsection{Proof of \upshape{Theorem~\ref{thm1}\ref{thm:part.ii})}} 
First we recall that
\begin{equation*}
\min\sess(\wh H_{\alpha}^{(+)})=m+E_{\eps}(\alpha) \quad \text{for all} \quad \alpha>0
\end{equation*}
and that
\begin{equation*}\label{min.sess.H-}
\min\sess(\wh H_{\alpha}^{(-)})=\begin{cases} 
2m-\eps & \mbox{if \eqref{weak.ultrared.regularity} holds, } m<2\eps \mbox{ and } \alpha\leq\alpha_{\text{cr}}, \\ m+E_{-\eps}(\alpha) & \mbox{otherwise}. 
\end{cases}
\end{equation*}
If \eqref{weak.ultrared.regularity} holds together with $m<2\eps$ and $0<\alpha\leq\alpha_{\text{cr}}$, then the claim follows by \eqref{disc.spec.relation} if we show only the relation $N\bigl(m+E_{\eps}(\alpha); \wh H_{\alpha}^{(+)}\bigr)<\infty$. To see this one should observe that $m+E_{\eps}(\alpha)<2m-\eps$ for all $\alpha>0$, see \eqref{m-e.big.E(alpha)}, and recall that the eigenvalues of a self-adjoint operator cannot accumulate below the bottom of its essential spectrum. Otherwise, the claim follows by \eqref{disc.spec.relation} if we show both of the relations
\begin{equation}\label{thm.disc.proof.aim}
N\bigl(m+E_{\sigma\eps}(\alpha); \wh H_{\alpha}^{(\sigma)}\bigr)<\infty, \quad \sigma=\pm.
\end{equation}
In fact, there is no loss of generality in focusing on the latter case as the proof we give below does not require any restrictions in the case $\sigma=+$. To this end, we fix the discrete variable $\sigma$ and denote $m+E_{\sigma\eps}(\alpha)$ by $z_0$ for notational convenience. Further, let $z\leq z_0$ be fixed for a moment. Since $(-\infty,z_0]\subset(-\infty,2m+\sigma\eps)=\rho(H^{(\sigma)}_{22})$, the integral operator $K^{(\sigma)}(z):L^2(\RR^d)\to L^2(\RR^d)$ is well-defined and compact (see Remark~\ref{rem:on.Delta.and.K}). Consider the decomposition
\begin{equation}
\frac{1}{\omega(k_1)+\omega(k_2)+\sigma\eps-z}=\Psi^{(\sigma)}_1(k_1,k_2;z)+\Psi^{(\sigma)}_2(k_1,k_2;z),
\end{equation}
where
\begin{equation}\label{kernel:k_1} \Psi^{(\sigma)}_1\!(k_1,k_2;z)\!:=\!\frac{1}{\omega(k_1)\!+\!m\!+\!\sigma\eps\!-\!z}+\frac{1}{\omega(k_2)\!+\!m\!+\!\sigma\eps\!-\!z}-\frac{1}{2m\!+\!\sigma\eps\!-\!z}
\end{equation}
and
\begin{equation}\label{kernel:k_2}
\Psi^{(\sigma)}_2\!(k_1,k_2;z)\!:=\!\frac{1}{\omega(k_1)\!+\!\omega(k_2)\!+\!\sigma\eps\!-\!z}-\Psi^{(\sigma)}_1\!(k_1,k_2;z),\quad k_1,k_2\in\RR^d.
\end{equation}
Let us denote by $K^{(\sigma)}_1(z)$ and $K^{(\sigma)}_2(z)$ the integral operators in $L^2(\RR^d)$ whose kernels are respectively the functions $(k_1,k_2)\mapsto\alpha^2\la(k_1)\ov{\la(k_2)}\Psi^{(\sigma)}_1(k_1,k_2;z)$ and $(k_1,k_2)\mapsto\alpha^2\la(k_1)\ov{\la(k_2)}\Psi^{(\sigma)}_2(k_1,k_2;z)$. Then, we have the corresponding decomposition
\begin{equation}\label{decomposition.of.K}
K^{(\sigma)}(z)=K^{(\sigma)}_1(z)+K^{(\sigma)}_2(z).
\end{equation} 
Since each term of the right-hand-side of \eqref{kernel:k_1} is uniformly bounded by the constant $\frac{1}{2m+\sigma\eps-z}$, we infer from \eqref{ass:la.in.L2} that 
$K^{(\sigma)}_1(z)$ is a well-defined rank-two operator in $L^2(\RR^d)$. In fact, its range coincides with the subspace of $L^2(\RR^d)$ spanned by the functions $\la$ and $\frac{\la}{\omega+m+\sigma\eps-z}$. Therefore, Lemma~\ref{lem:EV.count.1} combined with \cite[Theorem~IX.3.3]{Birman-Solomjak-87b} yield   
\begin{equation}\label{lem:estimates.EV.cf.101}
\begin{aligned}
N\bigl(z; \wh H_{\alpha}^{(\sigma)}\bigr) &= N\bigl(0; S_{\alpha}^{(\sigma)}(z)\bigr)=N\bigl(0; \Delta_{\alpha}^{(\sigma)}(z)-\alpha^2K^{(\sigma)}(z)\bigr)\\
&=N\bigl(0; \Delta_{\alpha}^{(\sigma)}(z)-\alpha^2K^{(\sigma)}_1(z)-\alpha^2K^{(\sigma)}_2(z)\bigr)\\
&\leq N\bigl(0; \Delta_{\alpha}^{(\sigma)}(z)-\alpha^2K^{(\sigma)}_2(z)\bigr)+2.
\end{aligned} 
\end{equation}
In particular, \eqref{thm.disc.proof.aim} will follow immediately if we show that 
\begin{equation}\label{lem:estimates.EV.cf.102}
N\bigl(0; \Delta_{\alpha}^{(\sigma)}(z_0)-\alpha^2K^{(\sigma)}_2(z_0)\bigr)<\infty.
\end{equation}
To this end, let $\Omega$ denote the complement of the level set of $\omega$ corresponding to~$m$, i.e. 
\begin{equation}
\Omega:=\{k\in\RR^d: \omega(k)\neq m\}.
\end{equation}
In view of Assumption~\ref{assumption}, it is clear that $\Omega$ is an open subset of $\RR^d$ with positive Lebesgue measure.

For $z\leq z_0$, we denote by $\Delta_{\alpha,\Omega}^{(\sigma)}(z)$ and $K_{2,\Omega}^{(\sigma)}(z)$ the restrictions of the operators $\Delta_{\alpha}^{(\sigma)}(z)$ and $K_{2}^{(\sigma)}(z)$ to $L^2(\Omega)$, respectively. Further, we denote by $\Delta_{\alpha,\Omega}^{(\sigma)}(\mydot;z)$ the restriction of the function $\Delta_{\alpha}^{(\sigma)}(\mydot;z)$ to $\Omega$ so that $\Delta_{\alpha,\Omega}^{(\sigma)}(z)$ is the operator of multiplication by the function $\Delta_{\alpha,\Omega}^{(\sigma)}(\mydot;z)$ in $L^2(\Omega)$. Since $\Phi^{(\sigma)}_{\alpha}(E_{\sigma\eps}(\alpha))=0$, by expressing $\Delta_{\alpha}^{(\sigma)}(k;z_0)$ as $\Delta_{\alpha}^{(\sigma)}(k;z_0)-\Phi^{(\sigma)}_{\alpha}(E_{\sigma\eps}(\alpha))$ and doing some elementary calculations, we obtain
\begin{equation*}
\Delta_{\alpha}^{(\sigma)}\!(k;z_0)\!=\!(\omega(k)-m)\bigg(1\!+\!\alpha^2\!\!\int\!\frac{|\la(q)|^2 \,\d q}{(\omega(q)\!+\!\sigma\eps\!-\!z_0\!+\!m)(\omega(q)\!+\!\omega(k)\!+\!\sigma\eps\!-\!z_0)}\bigg).
\end{equation*}
This implies that 
\begin{equation}\label{Delta.sim.omega}
\Delta_{\alpha,\Omega}^{(\sigma)}(k;z_0) \geq \omega(k)-m>0 \quad \text{for all} \quad k\in\Omega
\end{equation}
and $\Delta_{\alpha}^{(\sigma)}(\mydot;z_0)\equiv0$ on the level set of $\omega$ corresponding to $m$. In particular, the restriction of $\Delta_{\alpha}^{(\sigma)}(z_0)$ to $L^2(\RR^d\setminus\Omega)$ is the zero operator. Moreover, it is easy to see from \eqref{kernel:k_1}-\eqref{kernel:k_2} that the restriction of $K_{2}^{(\sigma)}(z_0)$ to $L^2(\RR^d\setminus\Omega)$ is the zero operator, too. That is why \eqref{lem:estimates.EV.cf.102} is equivalent to 
\begin{equation}\label{lem:estimates.EV.cf.1012}
N\bigl(0; \Delta_{\alpha,\Omega}^{(\sigma)}(z_0)-\alpha^2K_{2,\Omega}^{(\sigma)}(z_0)\bigr)<\infty.
\end{equation}
To show the latter, first we recall the monotonicity of the function $\Phi^{(\sigma)}_{\alpha}$ in $(-\infty, m+\sigma\eps)$. For all $z<z_0$, this together with \eqref{relation.between.P.and.Delta} implies that
\begin{equation}
\inf_{k\in\Omega}\Delta_{\alpha,\Omega}^{(\sigma)}(k;z) \geq \Phi^{(\sigma)}_{\alpha}(z-m)>\Phi^{(\sigma)}_{\alpha}(z_0-m)=\Phi^{(\sigma)}_{\alpha}(E_{\sigma\eps}(\alpha))=0.
\end{equation}
Hence, the multiplication operator $\Delta_{\alpha,\Omega}^{(\sigma)}(z)$ is positive for all $z<z_0$. Since the integral operator  $K_{2,\Omega}^{(\sigma)}(z):L^2(\Omega)\to L^2(\Omega)$ is well-defined and Hilbert-Schmidt, it follows that the operator 
\begin{equation}\label{B.Sch.op.T.eps}
T^{(\sigma)}_{\alpha}(z):=\bigl(\Delta_{\alpha,\Omega}^{(\sigma)}(z)\bigr)^{-1/2}K_{2,\Omega}^{(\sigma)}(z)\bigl(\Delta_{\alpha,\Omega}^{(\sigma)}(z)\bigr)^{-1/2}
\end{equation}
is also well-defined and Hilbert-Schmidt for all $z<z_0$. 

Next, we define $T^{(\sigma)}_{\alpha}(z_0)$ to be the integral operator in $L^2(\Omega)$ with the kernel
\begin{equation}\label{kern.sing.E.1}
\Theta^{(\sigma)}_{\alpha}(k_1,k_2):= \frac{\ov{\la(k_1)}\la(k_2)\Psi^{(\sigma)}_2(k_1,k_2;z_0)}{\sqrt{\Delta_{\alpha,\Omega}^{(\sigma)}(k_1;z_0)}\sqrt{\Delta_{\alpha,\Omega}^{(\sigma)}(k_2;z_0)}}, \quad k_1,k_2\in\Omega.
\end{equation}
We have the following elementary, yet quite important inequality
\begin{equation}\label{elementary.ineq}
0 \leq \frac{1}{a+b+c}-\frac{1}{a+c}-\frac{1}{b+c}+\frac{1}{c} \leq  \frac{\sqrt{ab}}{2c^2}
\end{equation}
which holds for all real numbers $a\geq0$, $b\geq0$ and $c>0$. Its proof is left to the ``Appendix''. Applying this inequality with $a=\omega(k_1)-m$, $b=\omega(k_2)-m$ and $c=2m+\sigma\eps-z_0$, we obtain the estimate
\begin{equation}\label{estimate.pointwise.on.Psi2}
0 \leq \Psi^{(\sigma)}_2(k_1,k_2;z_0) \leq \frac{1}{2(2m+\sigma\eps-z_0)^2}\sqrt{\omega(k_1)-m}\sqrt{\omega(k_2)-m}
\end{equation}
for all $k_1,k_2\in\Omega$. Using \eqref{Delta.sim.omega} and \eqref{estimate.pointwise.on.Psi2}, we can estimate the kernel $\Theta^{(\sigma)}_{\alpha}$ in \eqref{kern.sing.E.1} as follows
\begin{equation}
|\Theta^{(\sigma)}_{\alpha}(k_1,k_2)| \leq \frac{1}{2(2m+\sigma\eps-z_0)^2}|\la(k_1)||\la(k_2)|
\end{equation}
and thus we infer from \eqref{ass:la.in.L2} that $\Theta^{(\sigma)}_{\alpha}\in L^2(\Omega\times\Omega)$ with
\begin{equation}
\ds\|\Theta^{(\sigma)}_{\alpha}\|_{L^2(\Omega\times\Omega)} \leq \frac{1}{2(2m+\sigma\eps-z_0)^2}\|\la\|^2_{L^2(\RR^d)}<\infty.
\end{equation}
Therefore, $T^{(\sigma)}_{\alpha}(z_0):L^2(\Omega)\to L^2(\Omega)$ is Hilbert-Schmidt operator and the Lebesgue's dominated convergence theorem guarantees the left-continuity (with respect to the operator norm) of the operator function $T^{(\sigma)}_{\alpha}(\mydot)$ at $z_0$. 
This observation and the Weyl inequality 
\begin{equation*}
N\!\bigl(-1;\! -\alpha^2T^{(\sigma)}_{\alpha}\!(z)\bigr) \!\leq\! N\!\bigl(-0.5;\! -\alpha^2T^{(\sigma)}_{\alpha}\!(z_0)\bigr)+N\!\bigl(-0.5; \!\alpha^2T^{(\sigma)}_{\alpha}\!(z_0)\!-\!\alpha^2T^{(\sigma)}_{\alpha}\!(z)\bigr),
\end{equation*}
see \cite[Chapter~IX]{Birman-Solomjak-87b}, together with the left-continuity (with respect to the operator norm) of the operator functions $\Delta_{\alpha,\Omega}^{(\sigma)}(\mydot)$ and $K_{2,\Omega}^{(\sigma)}(\mydot)$ yield 
\begin{equation}
\begin{aligned}
N\bigl(0;\Delta_{\alpha,\Omega}^{(\sigma)}(z_0)-\alpha^2K_{2,\Omega}^{(\sigma)}(z_0)\bigr)&=\lim_{z\uparrow z_0}N\bigl(0; \Delta_{\alpha,\Omega}^{(\sigma)}(z)-\alpha^2K_{2,\Omega}^{(\sigma)}(z)\bigr)\\
&=\lim_{z\uparrow z_0}N\bigl(-1; -\alpha^2T^{(\sigma)}_{\alpha}(z)\bigr)\\
&\leq N\bigl(-0.5; -\alpha^2T^{(\sigma)}_{\alpha}(z_0)\bigr).
\end{aligned}
\end{equation}
However, $N\bigl(-0.5; -\alpha^2T^{(\sigma)}_{\alpha}(z_0)\bigr)$ must be a finite number because of the compactness of the operator $T^{(\sigma)}_{\alpha}(z_0)$, justifying \eqref{lem:estimates.EV.cf.1012}. \qed

\begin{remark}
	The finiteness of the discrete spectrum can be shown by even simpler way whenever the condition \eqref{weak.ultrared.regularity} is satisfied. We can proceed in the same way as in the above proof, the only exception being that there is no need for the ``special" decomposition \eqref{decomposition.of.K}. Instead of \eqref{B.Sch.op.T.eps}, we can consider the operator function 
	\begin{equation}
	\wh T^{(\sigma)}_{\alpha}(z):=\bigl(\Delta_{\alpha}^{(\sigma)}(z)\bigr)^{-1/2}K^{(\sigma)}(z)\bigl(\Delta_{\alpha}^{(\sigma)}(z)\bigr)^{-1/2}
	\end{equation}
	for $z<z_0$ and continuously extend it up to $z_0$ by defining $\wh T^{(\sigma)}_{\alpha}(z_0)$ to be the integral operator in $L^2(\RR^d)$ with the kernel
	\begin{equation*}\label{kern.sing.E}
	\wh\Theta^{(\sigma)}_{\alpha}(k_1,k_2):= \frac{\alpha^2\ov{\la(k_1)}\la(k_2)}{\sqrt{\Delta^{(\sigma)}_{\alpha}(k_1;z_0)}(\omega(k_1)+\omega(k_2)+\sigma\eps-z_0)\sqrt{\Delta^{(\sigma)}_{\alpha}(k_2;z_0)}}.
	\end{equation*}
	In fact, \eqref{Delta.sim.omega} together with the obvious inequality $\omega(k_1)+\omega(k_2)\geq 2m$ yields the estimate
	\begin{equation}
	|\wh\Theta^{(\sigma)}_{\alpha}(k_1,k_2)| \leq \frac{\alpha^2}{2m+\sigma\eps-z_0} \frac{|\la(k_1)|}{\sqrt{\omega(k_1)-m}}\frac{|\la(k_2)|}{\sqrt{\omega(k_2)-m}},
	\end{equation}
	which in turn implies $\wh\Theta^{(\sigma)}_{\alpha}\in L^2(\RR^d\times\RR^d)$ whenever \eqref{weak.ultrared.regularity} is satisfied.
\end{remark}

\begin{remark}
	The kernel splitting trick used in the proof of Theorem~\ref{thm1}ii) is inspired by~\cite{Ikromov-Sharipov-FAA-1998}.
\end{remark}

\subsection{Proof of \upshape{Theorem~\ref{thm2}}}\label{subsec:proof.of.thm2}
First we prove the asymptotic expansion \eqref{asymp.exp.bottom}. To this end, consider the function 
\begin{equation}
\psi(x,y)=-y-x\int \frac{|\la(q)|^2 \,\d q}{\omega(q)-xy+2\eps}
\end{equation}
for $(x,y)\in[0,\sqrt{m+2\eps})\times[0,\sqrt{m+2\eps})$. It is easy to see that $\psi$ is continuously differentiable and
\begin{equation}
\psi(0,0)=0, \quad \frac{\partial\psi}{\partial y}(0,0)=-1.
\end{equation}
Therefore, the implicit function theorem applies and yields the existence of a constant $\delta\in(0,\sqrt{m+2\eps})$ and a unique continuously differentiable function $\wh E:[0,\delta)\to\RR$ such that $\wh E(0)=0$ and $\psi(\alpha, \wh E(\alpha))=0$ for all $\alpha\in[0,\delta)$. Moreover, we have
\begin{equation}
\wh E'(0+)=\frac{\partial \psi}{\partial x}(0,0)=-\int\frac{|\la(q)|^2 \,\d q}{\omega(q)+2\eps}.
\end{equation} 
On the other hand, one can easily verify that
\begin{equation}
\Phi^{(+)}_{\alpha}(-\eps+\alpha\wh E(\alpha))=\alpha\psi(\alpha,\wh E(\alpha))=0 \quad \text{for all} \quad \alpha\in [0,\delta).
\end{equation}
Since $-\eps+\alpha\wh E(\alpha)<-\eps\leq\eps+m$ for all $\alpha\in [0,\delta)$ and $E_{\eps}(\alpha)$ is the unique zero of the function $\Phi^{(+)}_{\alpha}$ in the interval $(-\infty,\eps+m)$, we conclude that $E_{\eps}(\alpha)=-\eps+\alpha\wh E(\alpha)$ for all $\alpha\in [0,\delta)$. This implies that $E_{\eps}:[0,\delta)\to\RR$ is continuously differentiable and that the asymptotic expansion \eqref{asymp.exp.bottom} holds. 

To prove \eqref{form.ess.spec.H.weak.coupling}, we first note that 
\begin{equation}
\Phi^{(+)}_{\alpha}(m-\eps)=-m-\alpha^2 \int\frac{|\la(q)|^2\,\d q}{\omega(q)-m+2\eps}<0
\end{equation}
and thus 
\begin{equation}\label{E.alpha.small.than.m-eps}
E_\eps(\alpha)<m-\eps
\end{equation}
for all $\alpha>0$. Next, assume that \eqref{weak.ultrared.regularity} is satisfied and let $\alpha>0$ be such that \eqref{bound.on.small.alpha} holds. Using \eqref{E.alpha.small.than.m-eps} and \eqref{bound.on.small.alpha}, we obtain
\begin{equation*}
\begin{aligned}
\Phi^{(-)}_{\alpha}(E_{\eps}(\alpha))&=\Phi^{(-)}_{\alpha}(E_{\eps}(\alpha))-\Phi^{(+)}_{\alpha}(E_{\eps}(\alpha))\\
&=2\eps\bigg(1-\alpha^2 \int\frac{|\la(q)|^2\,\d q}{(\omega(q)-\eps-E_{\eps}(\alpha))(\omega(q)+\eps-E_{\eps}(\alpha))}\bigg)\\
&\geq 2\eps\bigg(1-\frac{\alpha^2}{2\eps} \int\frac{|\la(q)|^2\,\d q}{\omega(q)-m}\bigg)>0.
\end{aligned}
\end{equation*}
Therefore, for all $\alpha>0$ satisfying \eqref{bound.on.small.alpha} a possible zero $E_{-\eps}(\alpha)$ of $\Phi^{(-)}_{\alpha}(\mydot)$ lies on the right of $E_\eps(\alpha)$.

Now assume that \eqref{weak.ultrared.regularity} is not satisfied but $m>0$. Let us consider the function $\alpha\mapsto \Phi^{(-)}_{\alpha}(E_{\eps}(\alpha))$ for small $\alpha>0$. Using \eqref{asymp.exp.bottom}, we obtain 
\begin{equation}
\begin{aligned}
\Phi^{(-)}_{\alpha}(E_{\eps}(\alpha)) &\geq \eps-E_{\eps}(\alpha)-\alpha^2\int\frac{|\la(q)|^2\,\d q}{m-\eps-E_{\eps}(\alpha)}\\
&= 2\eps+\text{o}(\alpha)-\alpha^2\int\frac{|\la(q)|^2\,\d q}{m+\text{o}(\alpha)}=2\eps+\text{o}(\alpha), \quad \alpha\downarrow 0.
\end{aligned}
\end{equation}
Therefore, $E_\eps(\alpha)<E_{-\eps}(\alpha)$ for all sufficiently small $\alpha>0$.
\qed

\section{Appendix}\label{sec:appendix}
\subsection{Proof of the elementary inequality \eqref{elementary.ineq}}
The inequality holds trivially if one of $a$, $b$ is zero. So assume that $a>0$ and $b>0$. Elementary manipulations yield 
\begin{equation}\label{ap.ineq.1}
\frac{1}{a+b+c}-\frac{1}{a+c}-\frac{1}{b+c}+\frac{1}{c}=\frac{ab(a+b+2c)}{c(a+c)(b+c)(a+b+c)}.
\end{equation} 
Applying the AM-GM inequalities $\frac{a+c}{2}\geq \sqrt{ac}$ and $\frac{b+c}{2}\geq\sqrt{bc}$, we obtain 
\begin{equation}
\frac{ab(a+b+2c)}{c(a+c)(b+c)(a+b+c)}\leq \frac{\sqrt{ab}(a+b+2c)}{4c^2(a+b+c)}
<\frac{\sqrt{ab}}{2c^2}.
\end{equation}
The first inequality in \eqref{elementary.ineq} is a straightforward consequence of \eqref{ap.ineq.1}. \qed

\subsection{Spectrum of the one-boson system}\label{subsec:appendix.2}
The Hilbert space of the one-boson system is $\CC^2\otimes\cF_s^1$, where $\cF_s^1$ is the truncated Fock space
\begin{equation}
\cF_s^1:=\CC\oplus L^2(\RR^d).
\end{equation} 
For $f=\bigl(f^{(\sigma)}_0, f^{(\sigma)}_1\bigr)\in \CC^2\otimes\cF_s^1$, where $\sigma=\pm$ is the discrete variable, the one-boson Hamiltonian is given by the formal expression 
\begin{equation}\label{Hamiltonian:one-boson}
\begin{aligned}
(T_{\alpha}f)^{(\sigma)}_0 &= \sigma\eps f^{(\sigma)}_0 + \alpha\int \la(q)f^{(-\sigma)}_1\!(q) \,\d q,\\
(T_{\alpha}f)^{(\sigma)}_1\!(k) &=  (\sigma\eps+\omega(k))f^{(\sigma)}_1\!(k)+\alpha\la(k)f^{(-\sigma)}_0. 
\end{aligned}
\end{equation}
By means of the unitary transformation $V:\CC^2\otimes\cF_s^1\to\cF_s^1\oplus\cF_s^1$, defined by
\begin{equation}
V:
\begin{pmatrix}
\begin{pmatrix}
f_0^{(+)}\\
f_0^{(-)}
\end{pmatrix},
\begin{pmatrix}
f_1^{(+)}\\
f_1^{(-)}
\end{pmatrix}
\end{pmatrix}
\mapsto
\begin{pmatrix}
\begin{pmatrix}
f_0^{(+)}\\
f_1^{(-)}
\end{pmatrix},
\begin{pmatrix}
f_0^{(-)}\\
f_1^{(+)}
\end{pmatrix}
\end{pmatrix},
\end{equation}
we can block-diagonalize the Hamiltonian $T_{\alpha}$ in \eqref{Hamiltonian:one-boson}, that is, 
\begin{equation}\label{diagonalization.of.one-boson-Hamiltonian}
V^*T_{\alpha}V={\rm{diag}}\{T_{\alpha}^{(+)}, T_{\alpha}^{(-)}\}
\end{equation}
where 
\begin{equation}
T_{\alpha}^{(\sigma)}:=
\begin{pmatrix}
H^{(\sigma)}_{00} & \alpha H_{01}\\[0.5ex]
\alpha H_{10} & H^{(\sigma)}_{11}
\end{pmatrix}
\end{equation}
are $2\times2$ operator matrices acting in the truncated Fock space $\cF_s^1$ on the domain 
\begin{equation}\label{domain.of.T.sigma}
\Dom(T_{\alpha}^{(\sigma)}):=\CC\oplus\cH_1
\end{equation} 
with the weighted Hilbert space $\cH_1$ given in \eqref{Dom:H1}. The operator entries of $T_{\alpha}^{(\sigma)}$ are defined as in \eqref{def.op.entries.2}. By the same perturbation argument as in Section~\ref{subsec:prelimiaries} it follows that
\begin{equation}\label{ess.spec.T.alpha.pm}
\sess(T_{\alpha}^{(\sigma)})=\sigma(H^{(\sigma)}_{11})=[m-\sigma\eps,\infty).
\end{equation}
Furthermore, for all $z<m-\sigma\eps$ we have
\begin{equation}
\begin{aligned}
z\in\sigma(T_{\alpha}^{(\sigma)}) &\quad \Longleftrightarrow \quad 0\in\sigma\bigl(H^{(\sigma)}_{00}-z-\alpha^2H_{01}(H^{(\sigma)}_{11}-z)^{-1}H_{10}\bigr) \\
&\quad \Longleftrightarrow \quad \Phi^{(-\sigma)}_{\alpha}(z)=0,
\end{aligned}
\end{equation}
where the functions $\Phi^{(-\sigma)}_{\alpha}(\mydot)$ are defined as in \eqref{Phi}. In view of the analysis from Section~\ref{sec:main.results} on the zeros of $\Phi^{(\sigma)}_{\alpha}(\mydot)$, we conclude that
\begin{equation}\label{disc.spec.T.alpha.minus}
\sd(T_{\alpha}^{(-)})=\sigma_{\text{pp}}(T_{\alpha}^{(-)})=\{E_{\eps}(\alpha)\} 
\end{equation}
which holds for all $\alpha>0$ and independently of the condition \eqref{strong.ultrared.singularity}, whereas
\begin{equation}\label{disc.spec.T.alpha.plus}
\sd(T_{\alpha}^{(+)})=\sigma_{\text{pp}}(T_{\alpha}^{(+)})=
\begin{cases} \{E_{-\eps}(\alpha)\} & \mbox{in Cases 1 and 2a)},\\
\varnothing & \mbox{in Case 2b)}.
\end{cases}
\end{equation}
Altogether, combining \eqref{diagonalization.of.one-boson-Hamiltonian}, \eqref{ess.spec.T.alpha.pm}, \eqref{disc.spec.T.alpha.minus} and \eqref{disc.spec.T.alpha.plus}, we obtain 
\begin{equation}\label{ess.spec.one-boson.Hamiltonian}
\sess(T_{\alpha})=[m-\eps,\infty)
\end{equation}
and
\begin{equation}\label{spec.one-boson.Hamiltonian}
\sd(T_{\alpha})=\sigma_{\text{pp}}(T_{\alpha})=
\begin{cases} \{E_{-\eps}(\alpha), E_{\eps}(\alpha)\} & \mbox{in Cases 1 and 2a)},\\
\{E_{\eps}(\alpha)\} & \mbox{in Case 2b)}.
\end{cases}
\end{equation}	

{\small
	\bibliographystyle{acm}
	\bibliography{spin_boson_literature_20181108}
}
\end{document}